\documentclass[12pt,a4paper,reqno]{amsart}

\usepackage{a4wide}
\usepackage{enumerate}
\usepackage{graphicx}
\usepackage{yhmath}
\usepackage{amsmath,verbatim}
\usepackage{amssymb}
\usepackage{pgf,tikz}
\usetikzlibrary{arrows}
\usepackage[utf8]{inputenc}

\usepackage{hyperref}
\usepackage{cleveref}
\usepackage{physics}
\usepackage{amsthm}
\usepackage{color}
\usepackage{mathrsfs}
\usepackage{amsmath}
\usepackage{mathtools}
\usepackage{amsfonts}
\usepackage{amssymb}
\usepackage{bm}
\usepackage{physics}
\usepackage{enumitem}
\usepackage{tikz-cd}
\usepackage{a4wide}
\usepackage{cleveref}
\usepackage{esint}
\usepackage{nicefrac}
\usepackage{yfonts}
\usepackage{dirtytalk}

\newtheorem{thm}{Theorem}

\newtheorem{prop}[thm]{Proposition}
\newtoks\prt

\newtheorem{rem}[thm]{Remark}

\newtheorem{defn}[thm]{Definition}
\newtheorem{assp}[thm]{Assumption}
\newtheorem{example}[thm]{Example}

\newtheorem*{ex*}{Example}

\renewcommand{\epsilon}{\varepsilon}

\newcommand{\dist}{{\mathsf{d}}}

\newcommand{\XX}{{\mathsf{X}}}

\newcommand{\defeq}{\mathrel{\mathop:}=}

\newcommand{\RR}{\mathbb{R}}
\newcommand{\NN}{\mathbb{N}}

\newcommand{\EVI}{\mathrm{EVI}}

\newcommand{\C}{{\mathrm {C}}}
\newcommand{\AC}{{\mathrm {AC}}}

\newcommand{\CAT}{{\mathrm {CAT}}}

\DeclareMathOperator*{\argmin}{arg\,min}

\begin{document}
		\title[Stability of a class of action functionals]{Stability of a class of action functionals \\depending on convex functions}
		
		\author[L. Ambrosio]{Luigi Ambrosio}
		\address{L.~Ambrosio: Scuola Normale Superiore, Piazza dei Cavalieri 7, 56126 Pisa} 
		\email{\tt luigi.ambrosio@sns.it}
		
\author[C. Brena]{Camillo Brena}
\address{C.~Brena: Scuola Normale Superiore, Piazza dei Cavalieri 7, 56126 Pisa} 
\email{\tt camillo.brena@sns.it}

	\begin{abstract}
We study the stability of a class of action functionals induced by gradients of convex functions with respect to Mosco convergence, under mild assumptions on the underlying space.
\end{abstract}
\maketitle
\section*{Introduction}
The aim of this short note is to investigate the stability of action functionals of the form 
$$
\gamma\mapsto \int_0^1 \abs{\dot{\gamma}}^2(t)+\abs{\partial f}^2(\gamma(t))\dd{t},
$$ 
with respect to a suitable notion of convergence of the functions $f$.
Namely, we consider functionals defined on $\C([0,1],\XX)$ as
	\begin{equation}\notag
	\Theta^f_{x_0,x_1}(\gamma)\defeq
	\begin{cases}
		\displaystyle\int_0^1 \abs{\dot{\gamma}}^2(t)+\abs{\partial f}^2(\gamma(t))\dd{t}\quad&\text{if }  \gamma(0)=x_0,\gamma(1)=x_1;\\
		+\infty&\text{otherwise}
	\end{cases}
\end{equation}
and we study the $\Gamma$--convergence of $\Theta^{f^h}_{x^h_0,x^h_1}$ to $\Theta^{f}_{x_0,x_1}$ whenever $x_0^h\rightarrow x_0$, $x_1^h\rightarrow x_1$ and $f^h$ Mosco converges to $f$.

In the recent \cite{ABBGamma}, it has been proved that, in the Hilbertian setting, Mosco convergence of uniformly $\lambda$--convex functions $f^h$ to $f$ together with an uniform bound on the slopes $\abs{\partial f^h}(x_0^h)$ and $\abs{\partial f^h}(x_1^h)$ is enough to guarantee $\Gamma$--convergence of the respective functionals.  In the cited paper, the authors built recovery sequences relying on an interpolation lemma, which is in turn related to Minty’s trick and the structure of monotone operators. To follow such procedure a linear structure of the underlying space seems necessary, so that a generalization of their result to a wider class of spaces requires some work. Also, in the cited paper, the authors suggested that their result could be generalized to non-Hilbertian spaces, and that a natural instance would be the Wasserstein space $\mathcal{P}_2(\XX)$. 

The last observation of the paragraph above motivates then our note. We first generalize the $\Gamma$--convergence result to the case of complete $\CAT(0)$ spaces, in Theorem \ref{main}, and this is the content of Section \ref{sect1}. We then generalize the $\Gamma$--convergence result to the locally compact and geodesic metric space case, in Theorem \ref{main1}, with the uniform $\lambda$--convexity assumption taking the form of Assumption \ref{ass2}, and this is the content of Section \ref{sect2}. We point out that Assumption \ref{ass2} is already extensively used in the monograph \cite{AmbrosioGigliSavare08} as it is a crucial hypothesis to guarantee the existence of $\EVI_\lambda$ gradient flows in metric spaces. 
Moreover, if $\XX$ is an Hilbert space, functionals defined on the Wasserstein space $\mathcal{P}_2(\XX)$ that are $\lambda$--convex along generalized geodesics (\cite[Definition 9.2.4]{AmbrosioGigliSavare08}) satisfy Assumption \ref{ass2}, so we see that our result gives a positive answer to the problem suggested in \cite{ABBGamma}, at least for Wasserstein spaces built on compact Hilbert spaces. 
We see also with Theorem \ref{main2} that on any geodesic metric space, if $f^h=\epsilon_h f$, with $\epsilon_h\searrow 0$, where $f$ is regular enough and is continuous along the sequences $\{x_0^h\}_h$ and $\{x_1^h\}_h$, then the equiboundedness assumption on the slopes is not even needed to gain the $\Gamma$--convergence result, obtaining a different proof of a result contained in \cite{monsaingeon2020dynamical}.

The proofs of our main results, Theorem \ref{main} and Theorem \ref{main1} are similar in structure but use different tools, and are inspired by the techniques employed in \cite{ABBGamma} and \cite{monsaingeon2020dynamical} respectively. Now we briefly sketch the proofs of our main results, both to give an overview of the similarities of the procedure we use to what is already present in the literature and to highlight where are contained the main new ideas of this note. To keep this introduction short, we assume that the endpoints $x_0$ and $x_1$ are kept fixed. Of course, the only problem lies in building recovery sequences for a curve $\gamma$.

In Section \ref{sect1}, we start following \cite{ABBGamma} transforming $\gamma$ through the resolvent operator with respect to $f^h$, $J^{f^h}_{\tau_h}$, considering the curve $t\mapsto  J^{f^h}_{\tau_h}\gamma(t)$, for a suitable sequence $\tau_h\searrow 0$. Then we have to correct $\gamma^h$ at the endpoints. In \cite{ABBGamma} this is done using the interpolation lemma we have already mentioned above, here we use instead a variational interpolation applicable thanks to the crucial estimate contained in Proposition \ref{May98}.

For what concerns Section \ref{sect2}, we argue similarly as before, but replacing the resolvent operators with their \say{continuous version} given by the gradient flow trajectories $G^{f^h}_{\,\cdot\,}$. Also in \cite{monsaingeon2020dynamical} gradient flows were employed in a similar way, with the difference that they modified the curve in one step as $t\mapsto G^{f^h}_{g(t)}\gamma(t)$ for some continuous $g:[0,1]\rightarrow[0,1]$, whereas we first transform the whole curve into $t\mapsto G^{f^h}_{\tau_h}\gamma(t)$, for a suitable sequence $\tau_h\searrow 0$, and then we use the gradient flow again to recover the endpoints condition. The latter approach seems more flexible and allows us to treat other cases than the simpler one in which $f^h=\epsilon_h f$ for a sequence $\epsilon_h\searrow 0$.

\section*{Acknowledgements}
Work supported by the PRIN 2017 project ``Gradient flows, Optimal Transport and Metric Measure Structures".
The authors wish to thank Nicola Gigli for his valuable suggestions.
\section{Preliminaries}
Fix a metric space $(\XX,\dist)$ and $f:\XX\rightarrow\RR\cup\{+\infty\}$. We define the effective domain of $f$ as
\begin{equation}\notag
	D(f)\defeq\left\{x\in\XX: f(x)<+\infty\right\}.
\end{equation} We define also the descending slope,
\begin{equation}\notag
	\abs{\partial f}(x)\defeq
	\begin{cases}
	+\infty&\text{if $x\in \XX\setminus D(f)$}
	\\0&\text{if $x\in D(f)$ is isolated}\\
	\limsup_{y\rightarrow x}\frac{(f(y)-f(x))^-}{\dist(y,x)}\quad&\text{otherwise}.
	\end{cases}
\end{equation}
We will often require one of the following assumptions to hold.

\begin{assp}\cite[Assumption 2.4.5]{AmbrosioGigliSavare08}\label{ass1}
Let $(\XX,\dist)$ be a metric space and $f:\XX\rightarrow\RR\cup\{+\infty\}$. For $\lambda\in\RR$ assume that for any $x_0,x_1\in D(f)$ there exists a curve $\gamma:[0,1]\rightarrow\XX$ joining $x_0$ to $x_1$ such that the map 
$$x\mapsto f(x)+\frac{\dist(x,x_0)^2}{2 \tau}$$
is $(\lambda+\tau^{-1})$ convex along the curve $\gamma$ for every $\tau\in\left(0,\frac{1}{\lambda^-}\right)$.
\end{assp}
\begin{assp}\cite[Assumption 4.0.1]{AmbrosioGigliSavare08}\label{ass2}
	Let $(\XX,\dist)$ be a metric space and $f:\XX\rightarrow\RR\cup\{+\infty\}$. For $\lambda\in\RR$ assume that for any $y,x_0,x_1\in D(f)$ there exists a curve $\gamma:[0,1]\rightarrow\XX$ joining $x_0$ to $x_1$ such that the map 
	$$x\mapsto f(x)+\frac{\dist(x,y)^2}{2 \tau}$$
	is $(\lambda+\tau^{-1})$ convex along the curve $\gamma$ for every $\tau\in\left(0,\frac{1}{\lambda^-}\right)$.
\end{assp}
It is clear that Assumption \ref{ass2} is stronger than Assumption \ref{ass1}.
\begin{prop}\label{slopes}\cite[Theorem 2.4.9]{AmbrosioGigliSavare08}
Let $(\XX,\dist)$ be a metric space and let $f:\XX\rightarrow\RR\cup\{+\infty\}$ satisfy Assumption \ref{ass1} for some $\lambda\in\RR$. Then, for every $x\in D(f)$, \begin{equation}\notag
	\abs{\partial f}(x)=\sup_{y\ne x}\frac{\left(f(y)-f(x)-\frac{\lambda}{2}\dist(y,x)^2\right)^-}{\dist(y,x)}.
\end{equation}
\end{prop}
\begin{prop}\label{boundefrombelow}
Let $(\XX,\dist)$ be a metric space and let $f:\XX\rightarrow\RR\cup\{+\infty\}$ satisfy Assumption \ref{ass1} for some $\lambda\in\RR$. Assume moreover that there exists $\bar{x}\in D(f)$ such that
 \begin{equation}\notag
m\defeq\inf_{x\in \bar{B}_1(\bar{x})} f>-\infty.
\end{equation}
Then 
\begin{equation}\label{boundbelowc}
f(y)\ge \frac{\lambda}{2}\dist(y,\bar{x})^2+ \left(m-f(\bar{x})-\frac{\lambda^+}{2}\right)\dist(y,\bar{x}) + m.
\end{equation}
\end{prop}
\begin{proof}
If $\dist(y,\bar{x})\le 1$, the claim follows by a direct computation. If instead $\dist(y,\bar{x})>1$, we can conclude following similar computations to the ones used in \cite{AmbrosioGigliSavare08}.
Indeed, if $0<\tau<\frac{1}{\lambda^-}$ and $y$ is such that $\dist(y,\bar{x})>1$, we can take a curve (independent of $\tau$) $\gamma:[0,1]\rightarrow\XX$ joining $\bar{x}$ to y such that
\begin{equation}\notag
	\begin{split}
f(\gamma(t))+\frac{\dist(\gamma(t),\bar{x})^2}{2\tau}&\le (1-t)f(\bar{x})+t f(y)+t \frac{\dist(y,\bar{x})^2}{2\tau}-\frac{1}{2}\left(\frac{1}{\tau}+\lambda\right)t(1-t)\dist(y,\bar{x})^2
\\&= (1-t)f(\bar{x})+t f(y)-\frac{\lambda}{2}t(1-t)\dist(y,\bar{x})^2+\frac{1}{2\tau}t^2\dist(y,\bar{x})^2.
	\end{split}
\end{equation}
Notice that, as $\gamma$ is independent of $\tau$, we can multiply the inequality above by $\tau$ and let $\tau\searrow 0$ to obtain (here we assume $f(y)<+\infty$, otherwise there is nothing to show)
$$
\dist(\gamma(t),\bar{x})\le t \dist(y,\bar{x}).
$$
Therefore, if $\bar{t}\defeq\dist(y,\bar{x})^{-1}$, we have that $f(\gamma(t))\ge m$, so that if we neglect $\frac{\dist(\gamma(t),\bar{x})^2}{2\tau}$ and let $\tau\nearrow\frac{1}{\lambda^-}$,
\begin{equation}\notag
	\begin{split}
	m&\le (1-\bar{t})f(\bar{x})+\bar{t} f(y)-\frac{\lambda}{2}(1-\bar{t})\dist(y,\bar{x})+\frac{\lambda^-}{2}\\&=(1-\bar{t})f(\bar{x})+\bar{t} f(y)-\frac{\lambda}{2}\dist(y,\bar{x})+\frac{\lambda^+}{2}
	\end{split}
\end{equation}
so that the claim follows rearranging and noticing that 
$ f(\bar{x})\ge m$.
\end{proof}

Here we introduce our main object of investigation.
\begin{defn}\label{thetafunct}
	Let $(\XX,\dist)$ be a metric space, $x_0,x_1\in\XX$ and $f:\XX\rightarrow\RR\cup\{+\infty\}$. We define
	\begin{equation}\notag
\Theta^f_{x_0,x_1}:\C([0,1],\XX)\rightarrow\RR
	\end{equation} 
	\begin{equation}\notag
		\Theta^f_{x_0,x_1}(\gamma)\defeq
		\begin{cases}
			\displaystyle\int_0^1 \abs{\dot{\gamma}}^2(t)+\abs{\partial f}^2(\gamma(t))\dd{t}\quad&\text{if }\gamma\in\AC([0,1],\XX), \gamma(0)=x_0,\gamma(1)=x_1;\\
			+\infty&\text{otherwise}.
		\end{cases}
	\end{equation}
\end{defn}

If $(\XX,\dist)$ is a metric space and $f:\XX\rightarrow\RR\cup\{+\infty\}$, given $\tau\in(0,+\infty)$, we define the resolvent map
\begin{equation}\notag
	J^f_\tau x\defeq\argmin \left\{f(\,\cdot\,)+\frac{\dist(\,\cdot\,,x)^2}{2\tau}\right\}.
\end{equation}
Clearly, $J^f_\tau x$ can be empty as well as it can contain more than one element.

For what concerns the following result, the first inequality is \cite[Lemma 3.1.3]{AmbrosioGigliSavare08}, whereas the second can be showed following the proof of \cite[Theorem 2.1]{ABBGamma}.
\begin{prop}[{}]
Let $(\XX,\dist)$ be a metric space and $f:\XX\rightarrow\RR\cup\{+\infty\}$ be satisfying Assumption \ref{ass1} for some $\lambda\in\RR$.  
If $\tau\in\left(0,\frac{1}{\lambda^-}\right)$. Then, for every $u\in J_\tau^f x$,
\begin{equation}\label{bounds}
	\abs{\partial f}(u)\le \frac{\dist(u,x)}{\tau}\le\frac{1}{1+\lambda\tau}\abs{\partial f}(x).
\end{equation}
\end{prop}
\section{CAT$(0)$ setting}\label{sect1}
A $\CAT(0)$ space is a  geodesic metric space $(\XX,\dist)$ such that, for every triple of points $y,x_0,x_1\in\XX$ and constant speed geodesic $x:[0,1]\rightarrow\XX$ joining $x_0$ to $x_1$,  it holds 
$$
\dist(y,x_t)^2\le  (1 - t)\dist^2(y, x_0) + t\dist^2(y, x_1) - t(1 - t)\dist^2(x_0, x_1)
\quad\text{for every }t\in[0,1].$$
It easily follows that in $\CAT(0)$ spaces geodesics are unique, up to reparametrization: if $x,y:[0,1]\rightarrow\XX$ are constant speed geodesics, then
\begin{equation}
	\begin{split}
		\dist(y_t,x_t)^2&\le  (1 - t)\dist^2(y_t, x_0) + t\dist^2(y_t, x_1) - t(1 - t)\dist^2(x_0, x_1)\\&=\left((1-t) t^2+ t(1-t)^2- t(1-t)\right)\dist(y_0,y_1)^2=0.
	\end{split}
\end{equation}

A complete $\CAT(0)$ space is usually called an Hadamard space.
From now on $(\XX,\dist)$ will denote an Hadamard space. General references for this topic are \cite{Bac14,BH99}. It may be useful to know that the completion of a $\CAT(0)$ is still a $\CAT(0)$ space.

We can define a notion of weak convergence on Hadamard spaces (\cite[P.\ 58]{Bac14}) as follows.
First, if $\{x^h\}_h\subseteq\XX$ is any bounded sequence, one may prove that the map
$$ 
y\mapsto \limsup_h \dist(y,x^h)^2
$$
has a unique minimizer, that will be called the asymptotic centre of the sequence $\{x^h\}_h$.
We say then that $x^h$ weakly converges to $x$, and we write $x^h\rightharpoonup x$, if $x$ is the asymptotic centre of each subsequence $\{x^{h_k}\}_k$.
We remark that it is not known if the weak convergence is induced by a topology.
We state for future reference the following properties of weak convergence in Hadamard spaces.
\begin{prop}\label{propweaktopo}
Let $(\XX,\dist)$ be an Hadamard space. Then
\begin{enumerate}[label=\roman*)]
\item if $\{x^h\}_h\subseteq\XX$ is a bounded sequence, then it has a weakly convergent subsequence;
\item if $x^h\rightharpoonup x$ and $\dist(x^h,y)\rightarrow\dist(x,w)$ for some $w\in\XX$, then $x^h\rightarrow x$;
\item for every $w\in\XX$, $\dist(w,\,\cdot\,)$ is lower semicontinuous with respect to the weak convergence.
\end{enumerate}
\end{prop}
\begin{proof}
Items $i)$, $ii)$ and $iii)$ are taken respectively from \cite[Proposition 3.1.2, Proposition 3.1.6 and Corollary 3.2.4]{Bac14}. 

We give a shorter and self contained proof of item $iii)$ for the reader's convenience.
	We argue by contradiction assuming that for some $w\in\XX$ it holds that $$\liminf_h \dist(x^h,w)<\dist(x,w).$$ Possibly extracting subsequences and rescaling the metric, we can assume that  also $\dist(x,w)>1$ but $ \dist(x^h,w)\le 1$ for every $h$. Let $\gamma:[0,1]\rightarrow\XX$ be the (unique) constant speed geodesic joining $w$ to $x$ and let also $s\in[0,1]$ be the (unique) point for which $\dist(w,\gamma(s))=1$, i.e.\  $s=\dist(w,x)^{-1}$. 
	
	By the $\CAT(0)$ hypothesis, it holds that, for every $h$,
	\begin{equation}\notag
		\begin{split}
			\dist(x^h,\gamma(s))^2&\le (1-s)\dist(x^h,w)^2+s\dist(x^h,x)^2-s(1-s)\dist(w,x)^2\\&\le 1-s+s\dist(x^h,x)^2-(1-s)\dist(w,x)\le s\dist(x^h,x)^2
		\end{split}
	\end{equation}  
	that is a contradiction to the fact $x^h\rightharpoonup x$ as $s<1$ and $\gamma(s)\ne x$.
\end{proof}
We can see that this notion of weak topology coincides with the usual weak topology in Hilbert spaces. In order to do so, recall that if $x^h$ weakly converges to $x$ with respect to the weak topology induced by the Hilbert structure, then
$$
\dist(y,x)^2+\limsup_h  \dist( x,x^h)^2\le\limsup_h \dist( y,x^h)^2
$$
for any $y\in\XX$.
Then, to conclude, just recall that bounded sets are relatively compact with respect to both notions of weak convergence and that the two notions of weak convergence are stable with respect to the extraction of subsequences.

If $f:\XX\rightarrow\RR\cup\{+\infty\}$, we say that $f$ is $\lambda$--convex if for every constant speed geodesic $x:[0,1]\rightarrow\XX$, it holds
$$
f(x_t)\le (1-t)f(x_0)+t f(x_1)-\frac{\lambda}{2} t (1-t)\dist(x_0,x_1)^2\quad\text{for every }t\in[0,1].
$$
If $\lambda=0$, we just say that $f$ is convex.
In the sequel, we will often use the following fact, without notice: if $f$ is $\lambda$--convex, then $f$ satisfies Assumption \ref{ass2} (\textit{a fortiori}, Assumption \ref{ass1}) for that value of $\lambda$. 

\begin{prop}\label{May98}
Let $(\XX,\dist)$ be an Hadamard space and $f:\XX\rightarrow\RR\cup\{+\infty\}$  be proper, $\lambda$--convex and lower semicontinuous. 
	\begin{enumerate}[label=\roman*)]
\item If $\tau\in\left(0,\frac{1}{2\lambda^-}\right)$, then 
\begin{equation}\label{lipJ}
	J_\tau^f\text{ is single valued and $\frac{1}{\sqrt{1-2\lambda^-\tau}}$-Lipschitz}.
\end{equation}
\item If $0<\nu<\mu<\frac{1}{{2 \lambda^-}}$, then
\begin{equation}\label{liptau}
	\dist(J^f_\nu x,J^f_\mu x)\le \frac{1}{(1+\lambda\mu)\sqrt{1-2\lambda^-\nu}} (\mu-\nu)\abs{\partial f}(x)\quad\text{for every }x\in\XX.
\end{equation}
\end{enumerate}
\end{prop}
Concerning the proposition above, we remark that item $i)$ is proved in \cite[
Lemma 1.12]{Mayer1998}, whereas item $ii)$ follows combining $i)$, \eqref{bounds} and the so called resolvent identity (\cite[Lemma 1.10]{Mayer1998}), that reads as follows: if $0<\nu<\mu<\frac{1}{{ \lambda^-}}$ and $x\in\XX$,
\begin{equation}\label{resolvent}
	J^f_\mu x=J^f_\nu \left(\gamma\left(\frac{\nu}{\mu}\right)\right),
\end{equation} 
where $\gamma$ is the (unique) constant speed geodesic $\gamma:[0,1]\rightarrow\XX$ joining $J^f_\mu x$ to $x$.

\begin{prop}[{\cite[Proposition 2.2.17]{Bac14}}]\label{hasinf}
Let $(\XX,\dist)$ be an Hadamard space and $f:\XX\rightarrow\RR\cup\{+\infty\}$ be $\lambda$--convex and lower semicontinuous, with $\lambda>0$. Then $f$ attains its infimum (possibly $+\infty$). In particular, $f$ is bounded from below.
\end{prop}

\begin{defn}
Let $(\XX,\dist)$ be an Hadamard space. Let $\{f^h\}_h$ be functions from $\XX$ to $\RR\cup\{+\infty\}$ and let $f:\XX\rightarrow\RR\cup\{+\infty\}$. We say that $f_h$ Mosco converge to $f$ if the following two conditions hold:
\begin{enumerate}[label=\roman*)]
\item {$\Gamma$--$\liminf$ inequality:} if $\{x^h\}_h\subseteq\XX$ is such that $x^h\rightharpoonup x$, then $$ f(x)\le \liminf_h f^h(x^h);$$
\item {$\Gamma$--$\limsup$ inequality:} for every $x\in\XX$, there exists $\{x^h\}_h\subseteq\XX$ with $x^h\rightarrow x$ and $$\limsup_h f^h(x^h)\le f(x).$$
\end{enumerate}
\end{defn}
\begin{prop}\label{prop}
Let $(\XX,\dist)$ be an Hadamard space and $\lambda\in\RR$.
Let $\{f^h\}_h$, $f$ be proper $\lambda$--convex and lower semicontinuous functions from $\XX$ to $\RR\cup\{+\infty\}$ such that $f^h$ Mosco converge to $f$. Then, for any $x\in\XX$ and $\tau\in\left(0,\frac{1}{2 \lambda^-}\right)$, it holds 
\begin{equation}\label{convres}
J^{f^{h}}_\tau x \rightarrow J^f_\tau x.
\end{equation}

\end{prop}
\begin{proof}
Pick $z\in\XX$ with $f(z)<+\infty$ and, thanks to the $\Gamma$--$\limsup$ inequality, take a sequence $\{z^h\}_h\subseteq\XX$ such that $z^h\rightarrow z$ and $ f^h(z^h)\rightarrow f(z)$. In particular, we can assume $ f(z^h)<+\infty$ for every $h$. We write  $$G^h_\epsilon\defeq  f^h(\,\cdot\,)+\left(\frac{1}{2\tau}-\epsilon\right) \dist(\,\cdot\,,x)^2$$ where $\epsilon>0$ is small enough so that $1/\tau-2\epsilon>\lambda^-$, notice that
\begin{equation}\label{summ}
f^h(y)+\frac{1}{2\tau} \dist(y,x)^2=G^h_\epsilon(y)+\epsilon \dist(y,x)^2.
\end{equation}
Now, $G^h_\epsilon$ is $\lambda'$ convex and lower semicontinuous, with $\lambda'>0$, so that, as a consequence of Proposition \ref{boundefrombelow}, setting $m^h\defeq \inf_{\bar{B}_1 (z^h)} G^h_\epsilon$, we have that
$$G^h_\epsilon(y)\ge (m^h-G^h_\epsilon (z^h))\dist(y,z^h)+m^h.$$
Notice that thanks to items $ii)$ and $iii)$ of Proposition \ref{propweaktopo} and the $\Gamma$--$\liminf$ inequality, it holds, with the obvious definition for $G_\epsilon$,
$$\liminf_h\inf_{\bar{B}_1(z^h)} G_\epsilon^h \ge \inf_{\bar{B}_1(z)} G_\epsilon\ge \inf_{\XX} G_\epsilon>-\infty$$
where the last inequality follows from Proposition \ref{hasinf}.
Therefore, recalling \eqref{summ}, we easily obtain that
\begin{equation}\label{boundbelow}
f^h(y)+\frac{1}{2\tau} \dist(y,x)^2\ge C_1+C_2\dist(y,x) +\epsilon \dist(y,x)^2\quad\text{for every }h
\end{equation}
for some constants $C_1,C_2\in\RR$ independent of $h$.
As the $\Gamma$--$\limsup$ inequality easily implies that 
\begin{equation}\label{limsupinf}
\limsup_h\inf_y f^h(y)+\frac{1}{2\tau} \dist(y,x)^2\le  \inf_y f(y)+\frac{1}{2\tau} \dist(y,x)^2, 
\end{equation}
we easily obtain from \eqref{boundbelow} that the sequence $\{J_\tau^{f^h} x\}_h$ is bounded, hence, up to subsequences that we do not relabel, we have $J_\tau^{f^h} x\rightharpoonup w$ (item $i)$ of Proposition \ref{propweaktopo}).
Then, by the $\Gamma$--$\liminf$ inequality and item $iii)$ of Proposition \ref{propweaktopo}, $$\liminf_h  f^h(J_\tau^{f^h} x)+\frac{1}{2\tau} \dist(J_\tau^{f^h} x,x)^2\ge f(w)+\frac{1}{2\tau} \dist(w ,x)^2.$$
Taking into account \eqref{limsupinf} we have that $w=J^f_\tau x$ and 
$$\lim_h  f^h(J_\tau^{f^h} x)+\frac{1}{2\tau} \dist(J_\tau^{f^h} x,x)^2= f(J_\tau^f x)+\frac{1}{2\tau} \dist(J_\tau^f x ,x)^2,$$
in particular, $\dist(J_\tau^{f^h} x,x)\rightarrow \dist(J_\tau^f x,x)$. Strong convergence follows then from item $ii)$ of Proposition \ref{propweaktopo}.
\end{proof}

\begin{thm}\label{main}
Let $(\XX,\dist)$ be an Hadamard space and $\lambda\in\RR$.
Let $\{f^h\}_h$, $f$ be proper $\lambda$--convex and lower semicontinuous functions from $\XX$ to $\RR\cup\{+\infty\}$ such that $f^h$ Mosco converge to $f$. 
Let $\{x_0^h\}_h,\{x_1^h\}_h\subseteq\XX$ be two sequences with $x^h_0\rightarrow x_0$ and $x^h_1\rightarrow x_1$. Assume moreover that 
\begin{equation}\notag
\limsup_h \abs{\partial f^h}(x_0^h)<+\infty \quad\text{and}\quad \limsup_h \abs{\partial f^h}(x_1^h)<+\infty.
\end{equation}
Then $\Theta^{f^h}_{x_0^h,x_1^h}$ $\Gamma$--converge to $\Theta^f_{x_0,x_1}$ with respect to the $\C([0,1],\XX)$ topology.
\end{thm}
\begin{proof}
For the sake of simplicity, we write $\Theta^h$ instead of $\Theta^{f^h}_{x_0^h,x_1^h}$ and $\Theta$ instead of $\Theta^f_{x_0,x_1}$. Similar notation will be used for the resolvents $J_\tau^{f^h}$ and $J_\tau^f$.
In the sequel we let $C$ denote a constant; it may vary during the proof (but will be independent of $\tau$). 
\\\textsc{Step 1: $\Gamma$--$\liminf$ inequality}. The proof of this step is standard cf.\ \cite[Proof of Theorem 1.1]{ABBGamma} and is due to a simple lower semicontinuity argument.
\\\textsc{Step 2: $\Gamma$--$\limsup$ inequality}. 
We can assume that 
$$ \abs{\partial f^h}(x_0^h)+\abs{\partial f^h}(x_1^h)\le C\quad\text{for every }h.$$

Fix for the moment $\tau$ such that $0<\tau<\min\left\{\frac{1}{4 \lambda^-},\frac{1}{2}\right\}$ and keep in mind that the estimates in what follows will be independent of $\tau$. For the sake of simplicity, in the first part of this proof we will often to omit to write explicitly the dependence on $\tau$.
Notice also that $$ \frac{1}{1+\lambda\tau}+\frac{1}{\sqrt{1-2\lambda^-\tau}}\le C\quad\text{for every }\tau\in\left(0,\frac{1}{4\lambda^-}\right).$$

Fix $\gamma\in \C([0,1],\XX)$ such that $\Theta(\gamma)<+\infty$ and define $\gamma_\tau^h\defeq J_\tau^h\gamma$.
We compute, recalling \eqref{lipJ} and \eqref{bounds},
\begin{equation}\notag
	\begin{split}
		\int_0^1 |{\dot{\gamma^h}}|^2(t)+\abs{\partial f^h}^2(\gamma^h(t))\dd{t}\le \int_0^1 \frac{1}{1-2\lambda^-\tau} |{\dot{\gamma}}|^2(t)+\frac{\dist(\gamma(t),J^h_\tau\gamma(t))^2}{\tau^2}\dd{t}.
	\end{split}
\end{equation}
Thanks to the rough estimate, that follows from \eqref{lipJ} and \eqref{bounds},
\begin{equation}\notag
\dist(\gamma(t),J^h_\tau\gamma(t))\le \dist(\gamma(t),x_0^h)+\dist(x_0^h,J^h_\tau x_0^h)+\dist(J^h_\tau x_0^h,J^h_\tau\gamma(t))\le C
\end{equation}
we can use Proposition \ref{prop} together with dominated convergence to see that
\begin{equation}\notag
\limsup_h 	\int_0^1 |{\dot{\gamma^h}}|^2(t)+\abs{\partial f^h}^2(\gamma^h(t))\dd{t}\le \int_0^1 \frac{1}{1-2\lambda^-\tau} |{\dot{\gamma}}|^2(t)+\frac{\dist(\gamma(t),J_\tau\gamma(t))^2}{\tau^2}\dd{t}
\end{equation}
so that using \eqref{bounds},
\begin{equation}\label{bound1}
\limsup_h 	\int_0^1 |{\dot{\gamma^h}}|^2(t)+\abs{\partial f^h}^2(\gamma^h(t))\dd{t}\le (1+C\tau) \Theta(\gamma).
\end{equation}

Let $\tilde{\psi}^h_0:[0,\dist(x_0^h,x_0)]\rightarrow\XX$ be the (unique) constant speed geodesic joining $x_0^h$ to $x_0$ and define similarly $\tilde{\psi}^h_1$, but with the opposite ordering for $x_1^h$ to $x_1$. 
Let ${\psi}^h_0\defeq J_\tau^h \tilde{\psi}^h_0$ and define similarly ${\psi}^h_1$.  
Notice that using \eqref{bounds}, the fact that $\tilde{\psi}^h_0$ is $1$-Lipschitz and \eqref{bounds} again, we have that 
\begin{equation}\notag
\begin{split}
	\abs{\partial f^h}(\psi^h_0(t))&\le\frac{\dist(J_\tau^h{\tilde{\psi}}^h_0(t),\tilde{\psi}^h_0(t))}{\tau}\\&\le \frac{\dist(J_\tau^h{\tilde{\psi}}^h_0(t),J_\tau^h\tilde{\psi}^h_0(0))}{\tau}+\frac{\dist(J_\tau^h{\tilde{\psi}}^h_0(0),\tilde{\psi}^h_0(0))}{\tau}+\frac{\dist({\tilde{\psi}}^h_0(t),\tilde{\psi}^h_0(0))}{\tau}\\&\le
	\frac{C\dist(x_0^h,x_0)}{\tau}+C\abs{\partial f^h}(x_0^h)\le \frac{C}{\tau}.
\end{split}
\end{equation}
Therefore, recalling \eqref{lipJ}, we have that
\begin{equation}\label{bound3}
	\int_0^{\dist(x_0^h,x_0)} |{\dot{\psi_0^h}}|^2(t)+\abs{\partial f^h}^2(\psi_0^h(t))\dd{t}\le \frac{C}{\tau^2}\dist(x_0^h,x_0)
\end{equation}
and we also have a similar bound for $\psi_1^h$.

We consider the curves $\phi_0^h,\phi_1^h:[0,\tau]\rightarrow\XX$ $$\phi_0^h(s)\defeq J^h_s(x^h_0)\quad\text{and}\quad\phi_1^h(s)\defeq J^h_{1-s}(x^h_1).$$
Using \eqref{bounds} and the crucial  estimate \eqref{liptau} in item $ii)$ of Proposition \ref{May98}, we have that
\begin{equation}\label{bound2}
	\int_0^\tau |{\dot{\phi_0^h}}|^2(t)+\abs{\partial f^h}^2(\phi_0^h(t))\dd{t}\le C\tau
\end{equation}
and we also have a similar bound for $\phi_1^h$.

We let now $\tilde{\Phi}^h:[-\tau-\dist(x_0,x_0^h),1+\dist(x_1,x_1^h)+\tau]\rightarrow\XX$ be the curve obtained by concatenating $\phi_0^h$, $\psi_0^h$, $\gamma^h$, $\psi_1^h$, $\phi_1^h$. We define also $\Phi^h:[0,1]\rightarrow\XX$ be the curve obtained from $\tilde{\Phi}$ rescaling linearly the time, notice that $\Phi^h$ joins $x_0^h$ to $x_1^h$.

Notice that from the bounds \eqref{bound1}, \eqref{bound3}, \eqref{bound2} (and similar bounds for $\psi_1^h$ and $\phi_1^h$) it follows that 
\begin{equation}\label{thetalimsup}
\limsup_h \Theta^h(\Phi^h_\tau)\le (1+C\tau)^{C}\Theta(\gamma)+C\tau,
\end{equation}
where we made explicit the dependence on $\tau$ of the curves $\Phi^h_\tau$, that we omitted to write above for simplicity of notation.
We claim that 
\begin{equation}\label{goodcurveconv}
\lim_{\tau\searrow 0}\limsup_{h} \dist_\infty(\Phi_\tau^h,\gamma)=0,
\end{equation}
where $\dist_\infty$ is the distance inducing the $\C([0,1],\XX)$ topology. Assuming the claim to be true, using a diagonal argument together with \eqref{thetalimsup} and \eqref{goodcurveconv}  we can conclude the proof. Indeed, for every $0<\epsilon\ll 1$, take $0<\tau_\epsilon<\epsilon$ and $H_\epsilon$ such that if $h>H_\epsilon$,
\begin{equation}\label{tmp1}
\Theta^h(\Phi_{\tau_\epsilon}^h)\le (1+C\tau_\epsilon)^C\Theta(\gamma)+C\tau_\epsilon
\end{equation}
and 
\begin{equation}\label{tmp2}
\dist_\infty (\Phi_{\tau_\epsilon}^h,\gamma)<\epsilon.
\end{equation}
Taking $\epsilon_n\defeq n^{-1}$ we have a sequence $\{H_n=H_{\epsilon_n}\}_n$ that we can assume to be strictly increasing such that if $h>H_n$ then \eqref{tmp1} and \eqref{tmp2} hold with $\tau_{\epsilon_n}$ in place of $\tau_\epsilon$ and $\epsilon_n$ in place of $\epsilon$. We then take $\Psi^h\defeq \Phi^h_{\tau_n}$ if $H_n< h\le H_{n+1}$ and this is a suitable recovery sequence.

We show now \eqref{goodcurveconv}. Fix $\epsilon>0$. 
First notice that \eqref{thetalimsup} grants that, for every $\tau$, up to discarding finitely many curves (the number of the curves to be discarded may depend on $\tau$),    $\{\Phi_\tau^h\}_{h}$ are equi-Hölder continuous (where the Hölder constant can be assumed to be independent of $\tau$). Then, if $\tau$ is small enough,
\begin{equation}\label{notilde}
\limsup_h\dist(\Phi_\tau^h(t),\tilde{\Phi}_\tau^h(t))<\epsilon\quad\text{for every }t\in[0,1].
\end{equation}

Notice that the bound $\Theta(\gamma)<+\infty$ ensures that $\abs{\partial f}(\gamma(t))<+\infty$ a.e.\ then \eqref{bounds} shows that $J_\tau \gamma(t)\rightarrow\gamma(t)$ a.e. Now, recalling  \eqref{lipJ} and the bound $\Theta(\gamma)<+\infty$ again, we see that the curves $\{J_\tau\gamma\}_{\tau\le\tau_0}$ are equi-Hölder continuous.
Therefore, using an Arzelà–Ascoli argument we see that $$ \dist_\infty(J_\tau\gamma,\gamma)\rightarrow 0\quad\text{as }\tau\searrow 0.$$ Then, if $\tau$ is small enough, we have that 
\begin{equation}\label{uniformgamma}
\dist_\infty(\gamma,J_\tau\gamma)<\epsilon.
\end{equation}
Let $\tau$ be small enough so that \eqref{notilde} and \eqref{uniformgamma} hold.
Proposition \ref{prop} implies that $$\dist(\gamma(t),J^h_\tau\gamma(t))\rightarrow \dist(\gamma(t),J_\tau\gamma(t))<\epsilon \quad\text{for every $t$ as }h\rightarrow\infty.$$ Now, again by \eqref{lipJ} we see that $\{J^h_\tau\gamma(t)\}_h$ are equi-Hölder continuous so that also the maps $t\mapsto \dist(\gamma(t),J^h_\tau\gamma(t))$ are equi-Hölder continuous and therefore an Arzelà–Ascoli argument implies that the convergence is uniform, so that,
\begin{equation}\label{tmpfin}
\limsup_h\dist_\infty(\gamma,J_\tau^h\gamma)<2\epsilon.
\end{equation}
Taking into account that $\tilde{\Phi}^h_\tau(t)=\gamma^h_\tau(t)=J_\tau^h\gamma(t)$ for $t\in[0,1]$, we see that  \eqref{notilde} and \eqref{tmpfin} allow us to conclude the proof of the claim thanks to the subadditivity of the $\limsup$.
\end{proof}
\section{Gradient Flows Case}\label{sect2}
Recall that if $(\XX,\dist)$ is a metric space, $f:\XX\rightarrow\RR\cup\{+\infty\}$ and $x\in\XX$, an $\EVI_\lambda$ gradient flow trajectory for $f$ starting at $x$ is a map $G^f_{\cdot}x:[0,+\infty)\rightarrow D(f)$ such that the curve $x_t\defeq G^f_{t}x$ is locally absolutely continuous in $(0,+\infty)$ with\begin{equation}\notag
	\lim_{t\searrow 0} x_t=x_0=x
\end{equation} 
and satisfies, for every $v\in D(f)$,
\begin{equation}\notag
\frac{1}{2}\dv{t}\dist(x_t,v)^2+\frac{\lambda}{2}\dist(x_t,v)^2\le  f(v)-f(x_t)\quad\text{for a.e.\ }t\in (0,+\infty).
\end{equation}
We remark (\cite[Theorem 4.0.4]{AmbrosioGigliSavare08}) that if $(\XX,\dist)$ is complete and $f:\XX\rightarrow\RR\cup\{+\infty\}$ is proper, lower semicontinuous, bounded from below on some ball intersecting $D(f)$ and satisfies Assumption \ref{ass2} for some $\lambda\in\RR$, then, given any $x\in\overline{D(f)}$, there exists an $\EVI_\lambda$ gradient flow trajectory for $f$ starting at $x$.

In the following proposition we recall a few properties of $\EVI_\lambda$ gradient flows for future reference. We refer to \cite{AmbrosioGigliSavare08,MurSav}.
\begin{prop}
Let $(\XX,\dist)$ be a metric space, $\lambda\in\RR$, $f:\XX\rightarrow\RR\cup\{+\infty\}$ be a proper and lower semicontinuous function and $x_0,x_1\in\XX$. Assume that there exist two $\EVI_\lambda$ gradient flows trajectories $G^f_{\,\cdot\,}x_0,G^f_{\,\cdot\,}x_1$ starting respectively from $x_0,x_1$,
Then 
	\begin{enumerate}[label=\roman*)]
		\item It holds
\begin{equation}\label{gbound1}
\dist(G_t^f x_0,G_t^f x_1)\le e^{-\lambda t}\dist(x_0,x_1)\quad\text{for every $t\in [0,+\infty)$}.
\end{equation}
\item It holds
 \begin{equation}\label{gbound3}
|(G_t^f x_0)'|^2(t)=\abs{\partial f}^2(G_t^f x_0)=-\dv{t}(f\circ G_t^f x_0)\quad\text{for a.e.\ }t\in(0,+\infty).
\end{equation}
\item It holds
\begin{equation}\label{gbound2}
\frac{e^{\lambda t}-1}{\lambda t}\abs{\partial f} (G^f_t x_0)\le \frac{\dist(G^f_t x_0,x_0)}{t}\le e^{\lambda^- t}\abs{\partial f}(x_0)\quad \text{for every }t\in (0,+\infty),
\end{equation}
where $\frac{e^{\lambda t}-1}{\lambda t}$ has to be understood to be $1$ if $\lambda=0$.
	\end{enumerate}
\end{prop}
We remark that the second inequality of \eqref{gbound2} follows integrating the first equality in \eqref{gbound3}, after taking the square root of both sides, and taking into account that $t\mapsto e^{\lambda t}\abs{\partial f}(G^f_t x_0)$ is nonincreasing and right continuous and the first is precisely \cite[Equation (3.13)]{MurSav} with $u_t$ in place of $v$.

\begin{defn}
	Let $(\XX,\dist)$ be a metric space. Let $\{f^h\}_h$ be functions from $\XX$ to $\RR\cup\{+\infty\}$ and let $f:\XX\rightarrow\RR\cup\{+\infty\}$. We say that $f_h$ $\Gamma$--converge to $f$ if the following two conditions hold:
	\begin{enumerate}[label=\roman*)]
		\item {$\Gamma$--$\liminf$ inequality:} if $\{x^h\}_h\subseteq\XX$ is such that $x^h\rightarrow x$, then $$ f(x)\le \liminf_h f^h(x^h);$$
		\item {$\Gamma$--$\limsup$ inequality:} for every $x\in\XX$, there exists $\{x^h\}_h\subseteq\XX$ with $x^h\rightarrow x$ and $$\limsup_h f^h(x^h)\le f(x).$$
	\end{enumerate}
\end{defn}

\begin{prop}\label{uniformbound}
	Let $(\XX,\dist)$ be a proper metric space and let $\{f^h\}_h$, $f^h:\XX\rightarrow\RR\cup\{+\infty\}$, such that $f^h$ $\Gamma$--converges to $f:\XX\rightarrow\RR\cup\{+\infty\}$, with $f$ proper lower semicontinuous. Assume moreover that there exists $\lambda\in\RR$ such that $f^h$ and $f$ satisfy Assumption \ref{ass1} for that value of $\lambda$. Then, there exist ${C}_1,{C}_2\in\RR$  and $\bar{x}\in\XX$ (independent of $h$) such that 
	\begin{equation}\notag
		f^h(x)+ C_1\dist(x,\bar{x})^2\ge C_2\quad\text{for every }x\in\XX\text{ and }h\ge H_0
	\end{equation}
	for some $H_0\in\NN$ and the same inequality holds also with $f$ in place of $f^h$.
\end{prop}
\begin{proof}
	Take $\bar{x}\in\XX$ such that $f(\bar{x})<+\infty$. Notice that, being the space proper and $f$ lower semicontinuous, $\inf_{\bar{B}_1(\bar{x})} f>-\infty$. Then we can conclude exploiting Proposition \ref{boundefrombelow}.
\end{proof}

Combining \cite[Theorem 2.17]{DaneriSavare08}, Proposition \ref{uniformbound} and the existence results for $\EVI_\lambda$ gradient flows, we obtain the following.
\begin{thm}
	Let $(\XX,\dist)$ be a proper metric space and let $\{f^h\}_h$, $f^h:\XX\rightarrow\RR\cup\{+\infty\}$ be lower semicontinuous, such that $f^h$ $\Gamma$--converges to $f:\XX\rightarrow\RR\cup\{+\infty\}$, with $f$ proper and lower semicontinuous. Assume moreover that there exists $\lambda\in\RR$ such that $f^h$ and $f$ satisfy Assumption \ref{ass2} for that value of $\lambda$ and that $\overline{D(f^h)}=\XX$ for every $h$.
	Then, if we have $\{x^h\}_h\subseteq\XX$ such that $x^h\rightarrow x$, there exists the $\EVI_\lambda$ gradient flow trajectory for $f$ starting at $x$ and it satisfies
	\begin{equation}\notag
\lim_{h} G^{f^h}_{\,\cdot\,} x^h\rightarrow G^f_{\,\cdot\,}x\quad\text{locally uniformly on }(0,+\infty).
	\end{equation}
\end{thm}

With the discussion above in mind, we can adapt the proof of Theorem \ref{main}, replacing the resolvent map $J_\tau$ with the $\EVI_\lambda$ gradient flow $G_\tau$ to show the following.
\begin{thm}\label{main1}
	Let $(\XX,\dist)$ be a proper geodesic metric space and let $\{f^h\}_h$, $f^h:\XX\rightarrow\RR\cup\{+\infty\}$ be lower semicontinuous, such that $f^h$ $\Gamma$--converges to $f:\XX\rightarrow\RR\cup\{+\infty\}$, with $f$ proper and lower semicontinuous. Assume moreover that there exists $\lambda\in\RR$ such that $f^h$ and $f$ satisfy Assumption \ref{ass2} for that value of $\lambda$ and that $\overline{D(f^h)}=\XX$ for every $h$.
	Let $\{x_0^h\}_h,\{x_1^h\}_h\subseteq\XX$ be two sequences with $x^h_0\rightarrow x_0$ and $x^h_1\rightarrow x_1$. Assume moreover that 
	\begin{equation}\notag
		\limsup_h \abs{\partial f^h}(x_0^h)<+\infty \quad\text{and}\quad \limsup_h \abs{\partial f^h}(x_1^h)<+\infty.
	\end{equation}
	Then $\Theta^{f^h}_{x_0^h,x_1^h}$ $\Gamma$--converge to $\Theta^f_{x_0,x_1}$ with respect to the $\C([0,1],\XX)$ topology.
\end{thm}

\begin{thm}\label{main2}
	Let $(\XX,\dist)$ be a geodesic metric space and let  $f:\XX\rightarrow\RR\cup\{+\infty\}$ be a lower semicontinuous function 
	such that, for some $\lambda\in\RR$, for every $x\in\XX$, there exists an $\EVI_\lambda$ gradient flow trajectory for $f$ starting from $x$. Let $\{\epsilon_h\}_h\subseteq\RR$ with $\epsilon_h\searrow 0$. Let moreover $\{x_0^h\}_h\subseteq\XX$, $\{x_1^h\}_h\subseteq\XX$  be two sequences with $x_0^h\rightarrow x_0$, $x_1^h\rightarrow x_1$ and $$\limsup_h f(x^h_0)\le f(x_0)<\infty\quad\text{and}\quad\limsup_h f(x^h_1)\le f(x_1)<\infty.$$
	Then $\Theta^{\epsilon_h f}_{x_0,x_1}$ $\Gamma$--converge to $\Theta^0_{x_0,x_1}$ with respect to the $\C([0,1],\XX)$ topology.
\end{thm}
\begin{proof}
Of course, we only need to prove the $\Gamma$--$\limsup$ inequality. In the sequel we let $C$ denote a constant; it may vary during the proof.

We build a recovery sequence as follows.
Set $x^h_0(t)\defeq G^{f}_t x^h_0$. 
Now, if $\{t^h\}_h\subseteq\XX$ is such that $t^h\searrow 0$, then, using \eqref{gbound1},
\begin{equation}\label{converge}
\limsup_h \dist(x_0^h(t^h),x_0)\le \limsup_h \dist(G^f_{t^h} x_0^h,G^f_{t^h} x_0)+\limsup_h \dist(G^f_{t^h} x_0,x_0)=0.
\end{equation}
Taking into account the lower semicontinuity of $f$ together with \eqref{converge}, the convergence in the assumptions and \eqref{gbound3}, we have that
$$
\limsup_h \int_0^{\epsilon_h} |{\dot{x_0^h}}|^2(s)+\abs{\partial f}^2(x_0^h(s))\dd{s}=2\limsup_h\left( f(x^h_0)-2 f(x_0^h(\epsilon_h))\right)=0.$$
This implies that we can take $\{t_0^h\}_h\subseteq(0,+\infty)$ with $0<t_0^h<\epsilon_h$ and $$\limsup_h\abs{\partial\epsilon_h f}(x^h_0(t_0^h))\le C.$$
Clearly
$$\limsup_h \int_0^{t_0^h} |{\dot{x^h_0}}|^2(s)+\abs{\partial \epsilon_h f}^2(x^h_0(s))\dd{s}=0.$$
Set $\tilde{x}_0^h\defeq x^h_0(t_0^h)$. We argue in the same way for $x_1$ to obtain $\{x_1^h(\,\cdot\,)\}_h$ and $\{t_1^h\}_h$ and define $\tilde{x}_1^h\defeq x_1^h(t_1^h) $ accordingly. Clearly, $\tilde{x}_0^h\rightarrow x_0$ and $\tilde{x}_1^h\rightarrow x_1$.

Then, arguing as in the proof of Theorem \ref{main} (c.f.\ Theorem \ref{main1}), we obtain a recovery sequence for the $\Gamma$-- convergence of $\Theta^{\epsilon_h f}_{\tilde{x}_0^h,\tilde{x}_1^h}$ to $\Theta^0_{x_0,x_1}$.
It is then enough to join this recovery sequence with the curves $[0,t_0^h]\ni t\mapsto x_0^h (t)$ and  $[0,t_1^h]\ni t\mapsto x_0^h (t_1^h-t)$ and rescale linearly the time (again, see the proof of Theorem \ref{main})
\end{proof}
\begin{rem}
Theorem \ref{main1} and Theorem \ref{main2} need the space to be geodesic only to find geodesics joining $x_0^h$ to $x_0$ and $x_1^h$ to $x_1$. It is easily seen that this geodesics can be replaced by rectifiable curves whose length goes to $0$ as $h\rightarrow\infty$. 
In particular, if the endpoints are kept fixed, i.e.\ $x_0^h=x_0$ and $x_1^h=x_1$ for every $h$, the geodesic assumption  is unnecessary.
\end{rem}
\section{Examples}
We show, with a couple of trivial examples, that our assumptions are rather sharp.
\begin{example}
Let $(\XX,\dist)\defeq ([0,\infty),\dist_e)$ and \begin{equation}
f\defeq\begin{cases}
	\dfrac{1}{x^2}\quad&\text{if }x>0,\\
+\infty \quad&\text{if }x=0.
\end{cases}
\end{equation} 
Let also $\{\epsilon_h\}_h\subseteq  \RR$ with $\epsilon_h\searrow 0$. Choose $x_0^h\defeq \sqrt{\epsilon_h}$ and $x_1^h\defeq 1$ for every $h$, then $x_0=0$ and $x_1=1$.

It is clear that $\epsilon_h f$ $\Gamma$--converge to $0$. All the functions in consideration are convex, lower semicontinuous and admit $\EVI_0$ gradient flow trajectories. Also $$\limsup_h \abs{\epsilon_h f(x_0^h)}<+\infty\quad\text{and}\quad\limsup_h \abs{\epsilon_h f(x_1^h)}<+\infty,$$
but 
$$\limsup_h \abs{\partial \epsilon_h f(x_0^h)} =+\infty.$$
Notice that considering $\gamma\in\AC([0,1],\XX)$ defined as $\gamma(t)\defeq t$, we have that $\Theta_{x_0,x_1}^0(\gamma)=1$, but for any curve $\varphi \in\AC([0,1],\XX)$, we have that $\Theta_{x^h_0,x^h_1}^{\epsilon_h f}(\varphi)\ge 2^{-1}+(1-2\sqrt{\epsilon_h})^2$. This can be seen considering the portion of $\Theta_{x^h_0,x^h_1}^{\epsilon_h f}$ due to a part of $\varphi$ that joins $\sqrt{\epsilon_h}$ to $2\sqrt{\epsilon_h}$ and then a part of $\varphi$ that joins $2\sqrt{\epsilon_h}$ to $1$.
 Thus $\Theta_{x^h_0,x^h_1}^{f^h}$ does not $\Gamma$--converge to $\Theta_{x_0,x_1}^f$.
\end{example}

\begin{example}
Let $(\XX,\dist)\defeq ([0,\infty),\dist_e)$. Define
\begin{equation}
f^h(x)\defeq
\begin{cases}
1-h x\quad&\text{if }0\le x\le h^{-1},\\
0 \quad&\text{if }x>h^{-1};
\end{cases}
\end{equation}
$f\defeq 0.$
Choose $x_0^h\defeq 0$, $x_1^h\defeq 1$ for every $h$, then $x_0=0$ and $x_1=1$.

It is clear that $f^h$ $\Gamma$--converge to $f$. All the functions in consideration are convex, lower semicontinuous and admit $\EVI_0$ gradient flow trajectories. Also $$ \abs{f^h(x)}<1\quad\text{for every }x\in\XX.$$
but 
$$\limsup_h \abs{\partial f^h(x_0^h)} =+\infty.$$
Notice that considering $\gamma\in\AC([0,1],\XX)$ defined as $\gamma(t)\defeq t$, we have that $\Theta_{x_0,x_1}^f(\gamma)=1$, but for any curve $\varphi \in\AC([0,1],\XX)$, we have that $\Theta_{x^h_0,x^h_1}^{f^h}(\varphi)\ge 2$. This can be seen considering the portion of $\Theta_{x^h_0,x^h_1}^{ f^h}$ due to a part of $\varphi$ that joins $0$ to $h^{-1}$. Thus $\Theta_{x^h_0,x^h_1}^{f^h}$ does not $\Gamma$--converge to $\Theta_{x_0,x_1}^f$.
\end{example}

\bibliographystyle{alpha}
\bibliography{Biblio11}
\end{document}